\documentclass[12pt]{amsart}
\usepackage{amsmath,amssymb,amsbsy,amsfonts,latexsym,amsopn,amstext,
                                               amsxtra,euscript,amscd,bm}
                    
\usepackage{url}
\usepackage[colorlinks,linkcolor=blue,anchorcolor=blue,citecolor=blue]{hyperref}
\usepackage{color}

\begin{document}
\newtheorem{problem}{Problem}
\newtheorem{theorem}{Theorem}
\newtheorem{lemma}[theorem]{Lemma}
\newtheorem{claim}[theorem]{Claim}
\newtheorem{cor}[theorem]{Corollary}
\newtheorem{prop}[theorem]{Proposition}
\newtheorem{definition}{Definition}
\newtheorem{question}[theorem]{Question}

%%%%%%%%%%%%%%%%%%%%%%%%%
% Alphabet calligraphic %
%%%%%%%%%%%%%%%%%%%%%%%%%
\def\cA{{\mathcal A}}
\def\cB{{\mathcal B}}
\def\cC{{\mathcal C}}
\def\cD{{\mathcal D}}
\def\cE{{\mathcal E}}
\def\cF{{\mathcal F}}
\def\cG{{\mathcal G}}
\def\cH{{\mathcal H}}
\def\cI{{\mathcal I}}
\def\cJ{{\mathcal J}}
\def\cK{{\mathcal K}}
\def\cL{{\mathcal L}}
\def\cM{{\mathcal M}}
\def\cN{{\mathcal N}}
\def\cO{{\mathcal O}}
\def\cP{{\mathcal P}}
\def\cQ{{\mathcal Q}}
\def\cR{{\mathcal R}}
\def\cS{{\mathcal S}}
\def\cT{{\mathcal T}}
\def\cU{{\mathcal U}}
\def\cV{{\mathcal V}}
\def\cW{{\mathcal W}}
\def\cX{{\mathcal X}}
\def\cY{{\mathcal Y}}
\def\cZ{{\mathcal Z}}

%%%%%%%%%%%%%%%%%%%%%%%
% Alphabet blackboard %
%%%%%%%%%%%%%%%%%%%%%%%
\def\A{{\mathbb A}}
\def\B{{\mathbb B}}
\def\C{{\mathbb C}}
\def\D{{\mathbb D}}
\def\E{{\mathbb E}}
\def\F{{\mathbb F}}
\def\G{{\mathbb G}}
\def\I{{\mathbb I}}
\def\J{{\mathbb J}}
\def\K{{\mathbb K}}
\def\L{{\mathbb L}}
\def\M{{\mathbb M}}
\def\N{{\mathbb N}}
\def\O{{\mathbb O}}
\def\P{{\mathbb P}}
\def\Q{{\mathbb Q}}
\def\R{{\mathbb R}}
\def\S{{\mathbb S}}
\def\T{{\mathbb T}}
\def\U{{\mathbb U}}
\def\V{{\mathbb V}}
\def\W{{\mathbb W}}
\def\X{{\mathbb X}}
\def\Y{{\mathbb Y}}
\def\Z{{\mathbb Z}}

\def\ep{{\mathbf{e}}_p}
\def\em{{\mathbf{e}}_m}
\def\eq{{\mathbf{e}}_q}

\def\scr{\scriptstyle}
\def\\{\cr}
\def\({\left(}
\def\){\right)}
\def\[{\left[}
\def\]{\right]}
\def\<{\langle}
\def\>{\rangle}
\def\fl#1{\left\lfloor#1\right\rfloor}
\def\rf#1{\left\lceil#1\right\rceil}
\def\le{\leqslant}
\def\ge{\geqslant}
\def\eps{\varepsilon}
\def\mand{\qquad\mbox{and}\qquad}

\def\sssum{\mathop{\sum\ \sum\ \sum}}
\def\ssum{\mathop{\sum\, \sum}}
\def\ssumw{\mathop{\sum\qquad \sum}}

\def\vec#1{\mathbf{#1}}
\def\inv#1{\overline{#1}}
\def\num#1{\mathrm{num}(#1)}
\def\dist{\mathrm{dist}}

\def\fA{{\mathfrak A}}
\def\fB{{\mathfrak B}}
\def\fC{{\mathfrak C}}
\def\fU{{\mathfrak U}}
\def\fV{{\mathfrak V}}

\newcommand{\bflambda}{{\boldsymbol{\lambda}}}
\newcommand{\bfxi}{{\boldsymbol{\xi}}}
\newcommand{\bfrho}{{\boldsymbol{\rho}}}
\newcommand{\bfnu}{{\boldsymbol{\nu}}}

\def\GL{\mathrm{GL}}
\def\SL{\mathrm{SL}}

\def\Hba{\overline{\cH}_{a,m}}
\def\Hta{\widetilde{\cH}_{a,m}}
\def\Hb1{\overline{\cH}_{m}}
\def\Ht1{\widetilde{\cH}_{m}}

\def\flp#1{{\left\langle#1\right\rangle}_p}
\def\flm#1{{\left\langle#1\right\rangle}_m}
\def\dmod#1#2{\left\|#1\right\|_{#2}}
\def\dmodq#1{\left\|#1\right\|_q}

\def\Zm{\Z/m\Z}

\def\Err{{\mathbf{E}}}

\newcommand{\comm}[1]{\marginpar{%
\vskip-\baselineskip %raise the marginpar a bit
\raggedright\footnotesize
\itshape\hrule\smallskip#1\par\smallskip\hrule}}

\def\xxx{\vskip5pt\hrule\vskip5pt}

%%%%%%%%%%%%%%%%%%
%% PAPER BEGINS %%
%%%%%%%%%%%%%%%%%%

\title{Incomplete Gauss sums modulo primes}

 \author[B. Kerr] {Bryce Kerr}

\address{Department of Pure Mathematics, University of New South Wales,
Sydney, NSW 2052, Australia}
\email{bryce.kerr89@gmail.com}
%\title{\bf  Polynomial Vectors  of Small Height in 
%a Projective Spaces over Finite Fields}
\date{\today}
\pagenumbering{arabic}

%\keywords{??}
%\subjclass[2010]{??}

\begin{abstract} We obtain a new bound for incomplete Gauss sums modulo primes. Our argument falls under the framework of Vinogradov's method which we use to reduce the problem under consideration to bounding the number of solutions to two distinct  systems of congruences. The first is related to Vinogradov's mean value theorem, although the second does not appear to have been considered before. Our bound  improves on current results in the range $N\ge q^{2k^{-1/2}+O(k^{-3/2})}$.
 \end{abstract}
\maketitle
\section{Introduction}
The estimation of  exponential sums  of the form
\begin{align*}
\sum_{M<n\le M+N}e^{2\pi i f(n)},
\end{align*}
where $f$ is a polynomial  of large degree, is a common problem in number theory with a wide range of arithmetic consequences. This problem has been considered by a number of previous authors, see~\cite{Bom,HB,Par,Wo1} for recent progress on the estimation of such sums. See also~\cite{Fo} and~\cite[Chapter~8]{IwKo} for a brief overview of current results and techniques.

Let $a$ and $q$ be integers with $(a,q)=1$. In this paper we consider the problem of estimating sums of the form
\begin{align}
\label{eq:sums}
\sum_{M<n\le M+N}e_q(an^{k}),
\end{align}
where $e_q(z)$ is defined by
$$e^{2\pi i z/q}.$$
A consideration of the sums~\eqref{eq:sums} is motivated by Warings problem, which asks for the smallest integer $m$ such that all sufficiently large integers are the sum of at most $m$ $k$-th powers. The sums~\eqref{eq:sums} arise when treating the minor arcs in an application of the Circle Method to which we refer the reader to~\cite{Va,Wo2}. In such applications, one uses a bound of the form
\begin{align}
\label{eq:bound12345678}
\left|\sum_{M<n\le M+N}e_q(an^{k})\right|\le N^{1-1/\rho},
\end{align}
where $\rho$ depends on the integer $k$ and the size of $N$ relative to $q.$ The order of $\rho$ as a function of $k$ is an important factor in applications and it is desirable for $\rho$ to grow as slowly as possible for large $k$. 

 For large $k$, current methods for producing the sharpest bound for~\eqref{eq:bound12345678}  reduce the problem to estimating the number of solutions to a certain system of equations known as Vinogradov's mean value theorem, which we describe below. For these values of $k,$ the best known estimate for~\eqref{eq:bound12345678} relies on recent results of Bourgain, Demeter and Guth~\cite{BDG} (see for example~\cite[Lemma 2.1]{KuWo} and~\cite[Equation 4.5]{Fo}) and may be stated as follows. For $k\ge 3$ and $N\le q \le N^{k-1}$  we have 
\begin{align}
\label{eq:weylbest}
\left|\sum_{M<n\le M+N}e_q(an^{k})\right|\le N^{1-1/k(k-1)+o(1)},
\end{align}
where $q$ is an arbitrary integer and $(a,q)=1$. More precisely, the bound~\eqref{eq:weylbest} is sharpest known when $N$ is small relative to $q$. For example, when $q$ is prime one may use the Weil bounds to show 
\begin{align}
\label{eq:weilbound}
\left|\sum_{M<n\le M+N}e_q(an^{k})\right|\ll q^{1/2}\log{q},
\end{align}
 which is better than~\eqref{eq:weylbest} in the range $N\ge q^{1/2+1/k^2+O(1/k^3)}.$

 For integers $k,r$ and $V$ we let $J_{r,k}(V)$ count the number of solutions to the system of equations
\begin{align}
\label{eq:VMVT11111}
v_1^{j}+\dots+v_r^j=v_{r+1}^{j}+\dots+v_{2r}^{j}, \quad j=1,\dots,k,
\end{align}
with variables satisfying
$$1\le v_1,\dots,v_{2r}\le V.$$
Bounds for $J_{r,k}(V)$ are usually referred to as Vinogradov's mean value theorem and typically take the shape
\begin{align}
\label{eq:VMVT}
J_{r,k}(V)\le (1+V^{r-k(k+1)/2})V^{r+o(1)}.
\end{align}
The main conjecture for $J_{r,k}(V)$ is the statement that~\eqref{eq:VMVT} holds for all integers $r,k$ and $V$. Significant progress concerning bounds for $J_{r,k}(V)$ has been made by Wooley~\cite{Wo11,Wo12,Wo23,Wo3} and in particular Wooley~\cite{Wo31} has proven the main conjecture for $J_{r,k}(V)$ in the case $k=3$. More recently, Bourgain, Demeter and Guth~\cite{BDG} have proven the main conjecture for $J_{r,k}(V)$ when $k>3$. Combining the restults of Wooley~\cite{Wo31} for the case $k=3$  with those of Bourgain, Demeter and Guth for the case $k>3$,  for any integers $r,k$ and $V$ we have
\begin{align}
\label{eq:VMVTmain}
J_{r,k}(V)\le (1+V^{r-k(k+1)/2})V^{r+o(1)}.
\end{align}

In this paper we obtain a new bound for the sums~\eqref{eq:sums} when $q$ is prime. Our argument falls under the framework of Vinogradov's method which we use to reduce the problem to bounding two systems of congruences. The first concerns the number of solutions to the system
\begin{align}
\label{eq:VMVTc}
v_1^{j}+\dots+v_r^j\equiv v_{r+1}^{j}+\dots+v_{2r}^{j} \mod{q}, \quad j=1,\dots,k,
\end{align}
which has been considered by Karatsuba~\cite{Kar} who attributes the problem to Korobov. We bound the number of solutions to this system by reducing to Vinogradov's mean value theorem and applying results of Wooley~\cite{Wo31} and Bourgain, Demeter and Guth~\cite{BDG}. The second system of congruences (see Lemma~\ref{lem:mv11}) does not appear to have been considered before and we use some ideas based on Mordell~\cite{Mo}. We also note that a related system of equations has been considered by Arkhipov and  Karatsuba~\cite{ArKa}. 
\subsection*{Acknowledgement} The author would like to thank Igor Shparlinski for bringing the paper of Karatsuba~\cite{Kar} to the authors attention, for explaining the Russian version of Karatsuba's paper and for his comments on a preliminary version of the current paper.
\section{Main Results}
Our main result is as follows.
\begin{theorem}
\label{thm:main1}
Let $k\ge 3$ be an integer, $q$ a prime number, $a$ an integer with $(a,q)=1$ and  suppose $M$ and $N$ are integers with $N$ satisfying
\begin{align}
\label{eq:Nassumption}
N\le q^{1/2+1/(k^{1/2}+1)}.
\end{align}
Then we have 
\begin{align*}
\left|\sum_{M< n \le M+N}e_q(an^{k})\right| \le \left(\frac{q^{1/(k-2)^{1/2}}}{N}\right)^{2/k(k+1)}Nq^{o(1)}.
\end{align*}
\end{theorem}
We first note that the bound of Theorem~\ref{thm:main1} is nontrivial in the range
$$q^{(1+\varepsilon)/(k-2)^{1/2}}\le N,$$
and in this case we have a bound of the form
\begin{align*}
\left|\sum_{M< n \le M+N}e_q(an^{k})\right|\le N^{1-\rho+o(1)},
\end{align*}
where 
$$\rho=\frac{2\varepsilon}{(1+\varepsilon)k(k+1)}.$$
%$$\left|\sum_{M\le n \le M+N}e_q(an^{k})\right| \le N^{1-\rho}q^{o(1)},$$
%where $\rho$ is given by
%$$\left(1-\frac{\beta}{(k-2)^{1/2}} \right)\frac{1}{6k^{3/2}\log{k}+1},$$
%and $\beta$ is defined by
%$$N=q^{1/\beta}.$$
Comparing Theorem~\ref{thm:main1} with the bound~\eqref{eq:weylbest}, we see that Theorem~\ref{thm:main1} provides an improvement over~\eqref{eq:weylbest} in the range
$$N\ge q^{2(k-1)/(k-3)(k-2)^{1/2}}=q^{2k^{-1/2}+O(k^{-3/2})}.$$
\section{Preliminary results}
We first prove the following general inequality for systems of congruences.
\begin{lemma}
\label{lem:systemF}
Suppose $q$ is prime and $t,m$ and $\ell$ are integers. Let $$X\subseteq (\Z/q\Z)^{t},$$ be a set  and  $f=\{f_i\}_{i=1}^{m}$ a sequence of functions on $X$,
$$f_i : X\rightarrow \Z/q\Z, \quad i=1,\dots,m.$$
Let $\sigma=\{\sigma_i\}_{i=1}^{2\ell},$ be a sequence of numbers with each $\sigma_i \not \equiv 0 \mod{q}$ and $\lambda=\{\lambda_j\}_{j=1}^{m}$ a sequence of numbers with each $\lambda_j \in \Z/q\Z$. We let $I_{m,\ell}(f,X,\sigma,\lambda)$ denote the number of solutions to the system of equations
\begin{align}
\label{eq:systemF}
\sigma_1 f_j(x_1)+\dots+\sigma_{2\ell}f_j(x_{2\ell})\equiv \lambda_j \mod{q}, \quad j=1,\dots,m,
\end{align}
with variables $x_1,\dots,x_{2\ell}$ satisfying 
$$x_1,\dots,x_{2\ell}\in X.$$
 In the special case that $\sigma=\{(-1)^{i}\}_{i=1}^{2\ell}$ and each $\lambda_j=0$ we write
$$I_{m,\ell}(f,X,\{(-1)^{i}\}_{i=1}^{2\ell},0)=I_{m,\ell}(f,X).$$
For any $X,f,\sigma$ and $\lambda$ as above, we have
$$I_{m,\ell}(f,X,\sigma,\lambda)\le I_{m,\ell}(f,X).$$
\end{lemma}
\begin{proof}
Expanding the system~\eqref{eq:systemF} into additive characters, we see that 
\begin{align*}
& I_{m,\ell}(f,X,\sigma,\lambda) \\ & \ \ \ \ \ =\frac{1}{q^{m}}\sum_{\substack{1\le a_h \le q \\ 1\le h\le m}}\prod_{\substack{1\le i \le 2\ell}}\left(\sum_{\substack{x_i \in X}}e_q\left(\sigma_i\left(\sum_{j=1}^{m}a_j f_j(x_i)\right)\right) \right)\left(-\sum_{j=0}^{m}a_j\lambda_j \right),
\end{align*}
so that writing
$$f(a_1,\dots,a_{m})=\sum_{\substack{x\in X}}e_q\left(\sum_{j=1}^{m}a_j f_j(x)\right),$$
the above implies 
\begin{align*}
&  I_{m,\ell}(f,X,\sigma,\lambda)=\frac{1}{q^{m}}\sum_{\substack{1\le a_h \le q \\ 1\le h\le m}}\prod_{\substack{1\le i \le 2\ell}}f(\sigma_ia_1,\dots,\sigma_i a_{m})e_q\left(-\sum_{j=0}^{m}a_j\lambda_j \right),
\end{align*}
which we may bound by
\begin{align*}
& I_{m,\ell}(f,X,\sigma,\lambda) \le \frac{1}{q^{m}}\sum_{\substack{1\le a_h \le q \\ 1\le h\le m}}\prod_{\substack{1\le i \le 2\ell }}\left|f(\sigma_ia_1,\dots,\sigma_i a_{m})\right|.
\end{align*}
An application of H\"{o}lder's inequality gives
\begin{align*}
 I_{m,\ell}(f,X,\sigma,\lambda)\le \prod_{i=1}^{2\ell}\left(\frac{1}{q^{m}}\sum_{\substack{1\le a_h \le q \\ 1\le h\le m}}\left|f(\sigma_ia_1,\dots,\sigma_i a_{m})\right|^{2\ell} \right)^{1/2\ell},
\end{align*}
and hence
\begin{align*}
 I_{m,\ell}(f,X,\sigma,\lambda) &\le \prod_{i=1}^{2\ell}\left(\frac{1}{q^{m}}\sum_{\substack{1\le a_h \le q \\ 1\le h\le m}}\left|f(a_1,\dots,a_{m})\right|^{2\ell} \right)^{1/2\ell} \\
&=\frac{1}{q^{m}}\sum_{\substack{1\le a_h \le q \\ 1\le h\le m}}\left|f(a_1,\dots,a_{m})\right|^{2\ell}.
\end{align*}
The result follows since the term 
$$\frac{1}{q^{m}}\sum_{\substack{1\le a_h \le q \\ 1\le h\le m}}\left|f(a_1,\dots,a_{m})\right|^{2\ell},$$
counts the number of solutions to the system of congruences 
\begin{align*}
\sum_{i=1}^{2\ell}(-1)^{i}f_j(x_i)\equiv 0, \quad j=1,\dots,m,
\end{align*}
with variables $x_1,\dots,x_{2\ell}$ satisfying
$$x_1,\dots,x_{2\ell}\in X.$$
\end{proof}
The proof of the following uses the bound~\eqref{eq:VMVTmain}.
\begin{lemma}
\label{lem:VMVTq}
For integers $k,r,V$ and  $q$ we let $J_{r,k}(V;q)$ count the number of solutions to the system of congruences
\begin{align}
\label{eq:VMVTq}
v^j_1+\dots+v^j_{r}\equiv v_{r+1}^{j}+\dots+v_{2r}^{j} \mod{q}, \quad j=1,\dots,k,
\end{align}
with variables satisfying
$$1\le v_1,\dots,v_{2r} \le V.$$
Let $1\le m < k$ be an integer and suppose $V$ satisfies
\begin{align}
\label{eq:Vcond123}
q^{1/m}\ll V<\frac{q^{1/m}}{r}.
\end{align}
If $r\ge k(k+1)/2$ we have
$$J_{r,k}(V;q)\le V^{2r-km+m(m-1)/2+o(1)}.$$
\end{lemma}
\begin{proof}
For integers $\lambda_{m+1},\dots,\lambda_r$ we let 
$J_{r,k}(V,\lambda_{m+1},\dots,\lambda_{k})$ denote the number of solutions to the system of equations
\begin{align*}
v^j_1+\dots+v^j_{r}= v_{r+1}^{j}+\dots+v_{2r}^{j}, \quad j=1,\dots,m,
\end{align*}
\begin{align*}
v^j_1+\dots+v^j_{r}= v_{r+1}^{j}+\dots+v_{2r}^{j}+\lambda_j, \quad j=m+1,\dots,k,
\end{align*}
with variables $v_1,\dots,v_{2r}$ satisfying
$$1\le v_1,\dots,v_{2r} \le V.$$
Expressing $J_{r,k}(V,\lambda_{m+1},\dots,\lambda_{k})$ via additive characters we see that
\begin{align*}
&J_{r,k}(V,\lambda_{m+1},\dots,\lambda_{k}) \\ & \quad =\int_{0}^{1}\dots \int_{0}^{1}\left|\sum_{1\le v \le V}e^{2\pi i(\alpha_1 v+\dots+\alpha_k v^{k})} \right|^{2r}e^{-2\pi i(\alpha_{m+1}\lambda_{m+1}+\dots+\alpha_k \lambda_{k})}d\alpha_1\dots d\alpha_k,
\end{align*} 
and hence
\begin{align*}
J_{r,k}(V,\lambda_{m+1},\dots,\lambda_{k})\le \int_{0}^{1}\dots \int_{0}^{1}\left|\sum_{1\le v \le V}e^{2\pi i(\alpha_1 v+\dots+\alpha_k v^{k})} \right|^{2r}d\alpha_1\dots d\alpha_k,
\end{align*}
which implies
\begin{align}
\label{eq:JJJJJ}
J_{r,k}(V,\lambda_{m+1},\dots,\lambda_{k})\le J_{r,k}(V).
\end{align}
Using the assumption~\eqref{eq:Vcond123}, we have
\begin{align*}
J_{r,k}(V;q)=\sum_{\substack{|\lambda_{j}|\le rV^{j}/q \\ m+1\le j \le k}}J_{r,k}(V,\lambda_{m+1}q,\dots,\lambda_{k}q),
\end{align*}
and hence by~\eqref{eq:JJJJJ}
\begin{align*}
J_{r,k}(V;q)\ll \left(\prod_{j=m+1}^{k}\frac{V^j}{q}\right)J_{r,k}(V).
\end{align*}
Since $r\ge k(k+1)/2,$  an application of~\eqref{eq:VMVTmain} gives
\begin{align*}
J_{r,k}(V;q)\ll \left(\prod_{j=m+1}^{k}\frac{V^j}{q}\right)V^{2r-k(k+1)/2+o(1)},
\end{align*}
which on recalling~\eqref{eq:Vcond123} simplifies to
\begin{align*}
J_{r,k}(V;q)\ll V^{2r-km+m(m-1)/2+o(1)}.
\end{align*}
\end{proof}
For a proof of the following see~\cite{FI}. See also~\cite{ACZ,GG,Ke} for realted and more precise results.
\begin{lemma}
\label{lem:multcong}
Let $q$ be prime and suppose $M,N$ and $U$ are integers with $N$ and $U$ satisfying
$$NU\le q , \quad U\le N.$$
The number of solutions to the congruence
$$n_1u_1\equiv n_2u_2 \mod{q},$$
with variables satisfying
$$M< n_1,n_2\le M+N, \quad 1\le u_1,u_2 \le U,$$
is bounded by $$O(NU\log{q}).$$
\end{lemma}
\begin{lemma}
\label{lem:mv11}
Let $q$ be prime and $M,N$ and $U$  integers with $N$ and $U$ satisfying
$$NU\le q, \quad U\le N.$$
Let $\ell \ge 1$ be an integer and suppose $k$ is an integer satisfying
$$2\ell \le k.$$
Let 
$I_{k,\ell}(N,U)$ denote the number of solutions to the system of congruences
\begin{align}
\label{eq:system1}
n_1^{j}u_1^{k-j}+\dots+n_{\ell}^ju_{\ell}^{k-j}\equiv n_{\ell+1}^{j}u_{\ell+1}^{k-j}+\dots+n_{2\ell}^ju_{2\ell}^{k-j} \mod q, \quad j=0,\dots,2\ell-1,
\end{align}
in variables $n_1,\dots,n_{2\ell},u_1,\dots,u_{2\ell}$ satisfying
\begin{align}
\label{eq:conditions1}
M< n_1,\dots,n_{2\ell} \le M+N, \quad 1\le u_1,\dots,u_{2\ell}\le U.
\end{align}
Then we have
\begin{align*}
I_{k,\ell}(N,U)\ll (NU)^{\ell}(\log{q})^{\ell-1}.
\end{align*}
\end{lemma}
\begin{proof}
We fix an integer $k$ and first consider the case $\ell=1$. Recalling that $I_{k,1}(N,U)$ counts the number of solutions to the congruences
\begin{align*}
u_1^{k}\equiv u_2^{k} \mod{q} \quad \text{and} \quad  u_1^{k-1}n_1\equiv u_2^{k-1}n_2 \mod{q},
\end{align*}
in variables 
$$M\le n_1,n_2 \le M+N, \quad 1\le u_1,u_2 \le U,$$
we fix a value of $u_1$ which determines $u_2$  with at most $k$ choices. Next we fix a value of $n_1$ which determines $n_2$ with at most $1$ choice. This implies that 
\begin{align}
\label{eq:Ik1}
I_{k,1}(N,U)\ll NU.
\end{align}
\par
We next assume there exists some integer $m$ such that 
\begin{align}
\label{eq:inequalityassumption}
I_{k,m}(N,U)\gg (NU)^{m}(\log{q})^{m-1},
\end{align}
and out of all integers $m$ satisfying~\eqref{eq:inequalityassumption} suppose $\ell$ is the smallest, so that
\begin{align}
\label{eq:inequalityassumption1}
I_{k,\ell}(N,U)\gg (NU)^{\ell}(\log{q})^{\ell-1}.
\end{align}
 Considering~\eqref{eq:Ik1}, we see that
$$\ell \ge 2.$$
\par 
Suppose the points $n=(n_1,\dots, n_{2\ell})$ and $u=(u_1,\dots,u_{2\ell})$ are a solution to~\eqref{eq:system1} satisfying~\eqref{eq:conditions1}. Let $$H(u)\subseteq \F_q^{2\ell},$$ denote the hyperplane
$$H(u)=\{ (z_1,\dots,z_{2\ell})\in \F_q^{2\ell} : z_1u_1^{k}+\dots-z_{2\ell}u_{2\ell}^{k}\equiv 0 \mod q\}.$$
Considering the system~\eqref{eq:system1}, we see that the $2\ell$ points 
$$p_j(n,u)=\left(\left(\frac{n_1}{u_1}\right)^{j}, \dots,\left(\frac{n_{2\ell}}{u_{2\ell}}\right)^{j} \right), \quad j=0,\dots,2\ell-1,$$
all lie on $H(u)$. Since $H(u)$ has dimension $2\ell-1$, this implies there exists some nontrivial linear relation among the $p_j(n,u)$. In particular, there exists $\alpha_0,\dots,\alpha_{2\ell-1}\in \F_q,$ with at least one $\alpha_i \not \equiv 0 \mod {q}$  such that 
\begin{align}
\label{eq:system1122}
\alpha_0 p_0(n,u)+\dots+\alpha_{2\ell-1}p_{2\ell-1}(n,u)\equiv 0 \mod q.
\end{align}
The equation~\eqref{eq:system1122} implies that the  Vandermonde matrix of the points
$$\frac{n_1}{u_1},\dots,\frac{n_{2\ell}}{u_{2\ell}},$$
is singular mod $q$ and hence there exists integers $1\le r,s \le 2\ell$ with $r\neq s$ such that
\begin{align}
\label{eq:furthurcondition1}
\frac{n_r}{u_r}\equiv \frac{n_s}{u_s} \mod q.
\end{align}

Letting $I_{k,\ell}(N,U,r,s)$ denote the number of solutions to the system~\eqref{eq:system1} with variables satisfying~\eqref{eq:conditions1} subject to the furthur condition~\eqref{eq:furthurcondition1},
since the pair $r,s$ can take at most $4\ell^{2}$ values, we see that 
\begin{align}
\label{eq:Iinduction1}
I_{k,\ell}(N,U)\ll I_{k,\ell}(N,U,r,s),
\end{align}
for some pair $r,s$ with $r\neq s$.
 Considering $I_{k,\ell}(N,U,r,s)$, there exists some sequence of numbers $\sigma_i\in \{-1,1\}$ for $1\le i \le 2\ell$ such that $I_{k,\ell}(N,U,r,s)$ is equal to the number of solutions to the system of congruences
\begin{align}
\label{eq:I0def}
\sum_{\substack{1\le i \le 2\ell \\ i\neq r,s}}\sigma_i n_i^{j}u_i^{k-j} \equiv \sigma_rn_r^{j}u_r^{k-j}+\sigma_s n_s^{j}u_s^{k-j} \mod q, \quad  j=0,\dots,2\ell-1,
\end{align}
with variables satisfying~\eqref{eq:conditions1} and~\eqref{eq:furthurcondition1}.
\par
 For each fixed $n_r,n_s,u_r,u_s,$ we let 
$I_0(n_r,n_s,u_r,u_s)$ count the number of solutions to the system~\eqref{eq:I0def} with varaibles satisfying
\begin{align}
M< n_i \le M+N, \quad 1\le u_i \le U, \quad i=1,\dots,2\ell, \quad i\neq r,s,
\end{align}
so that the above implies
\begin{align}
\label{eq:Iinduction}
I_{k,\ell}(N,U) \ll \sum_{\substack{M< n_r,n_s \le M+N \\ 1\le u_r,u_s \le U \\ n_ru_s \equiv n_su_r \mod q}}I_0(n_r,n_s,u_r,u_s).
\end{align}
Letting $I'_0$ denote the number of solutions to the system of congruences
\begin{align}
\label{eq:system2}
n_1^{j}u_1^{k-j}+\dots+n_{\ell-1}^ju_{\ell-1}^{k-j}\equiv n_{\ell}^{j}u_{\ell}^{k-j}+\dots+n_{2\ell-2}^ju_{2\ell-2}^{k-j} \mod q, \quad j=0,\dots,2\ell-1,
\end{align}
in variables $n_1,\dots,n_{2\ell-2},u_1,\dots,u_{2\ell-2}$ satisfying
\begin{align}
\label{eq:conditions2}
M< n_1,\dots,n_{2\ell-2} \le M+N, \quad 1\le u_1,\dots,u_{2\ell-2}\le U,
\end{align}
an application of Lemma~\ref{lem:systemF} gives
\begin{align}
\label{eq:I011111}
I_0(n_r,n_s,u_r,u_s)\le I_0',
\end{align}
for any $n_r,n_s,u_r$ and $u_s$.
Considering only equations corresponding to $$j=0,\dots,2(\ell-1)-1,$$ in the system~\eqref{eq:system2}, we see that
$$I_0'\le I_{k,\ell-1}(N,U),$$
and hence by~\eqref{eq:I011111} 
$$I_0(n_r,n_s,u_r,u_s)\ll I_{k,\ell-1}(N,U).$$
Combining the above with~\eqref{eq:Iinduction} we get 
\begin{align*}
I_{k,\ell}(N,U) \ll \left(\sum_{\substack{M< n_r,n_s \le M+N \\ 1\le u_r,u_s \le U \\ n_ru_s \equiv n_su_r \mod q}}1\right)I_{k,\ell-1}(N,U).
\end{align*}
Since the term
$$\sum_{\substack{M< n_r,n_s \le M+N \\ 1\le u_r,u_s \le U \\ n_ru_s \equiv n_su_r \mod q}}1,$$
counts the number of solutions to the congruence
$$n_ru_s \equiv n_su_r \mod{q},$$
in variables $n_r,n_s,u_r,u_s$ satisfying
$$M< n_r,n_s\le M+N, \quad 1\le u_r,u_s \le U,$$
an application of  Lemma~\ref{lem:multcong} gives 
$$\sum_{\substack{M< n_r,n_s \le M+N \\ 1\le u_r,u_s \le U \\ n_ru_s \equiv n_su_r \mod q}}1\ll NU\log{q},$$
and hence 
\begin{align*}
I_{k,\ell}(N,U) \ll I_{k,\ell-1}(N,U)NU\log{q}.
\end{align*}
Combining the above with~\eqref{eq:inequalityassumption1}, we see that
$$I_{k,\ell-1}(N,U)\gg (NU)^{\ell-1}(\log{q})^{\ell-2},$$
contradicting the assumption that $\ell$ is the smallest integer satsifying~\eqref{eq:inequalityassumption}.
\end{proof}
\section{Proof of Theorem~\ref{thm:main1}}
We proceed by induction on $N$ and fix an arbitrarily small $\varepsilon$.  Since the bound of Theorem~\ref{thm:main1} is trivial provided $N\le q^{1/(k-2)^{1/2}}$, this forms the basis of the induction. We formulate our induction hypothesis as follows. Let $M$ be an arbitrarty integer and suppose for all integers $K\le N-1$ we have
\begin{align*}
\left|\sum_{M< n \le M+K}e_q(an^{k}) \right|\le \left(\frac{q^{1/(k-2)^{1/2}}}{K}\right)^{2/k(k+1)}Kq^{\varepsilon},
\end{align*}
uniformly over $M$.
Using the above hypothesis, we aim to show
\begin{align}
\label{eq:inductionSum}
\left|\sum_{M< n \le M+N}e_q(an^{k}) \right|\le \left(\frac{q^{1/(k-2)^{1/2}}}{N}\right)^{2/k(k+1)}Nq^{\varepsilon}.
\end{align}
We define the integers $\ell,m$ and $r$ by
\begin{align}
\label{eq:elldef}
\ell = \left\lfloor \frac{k}{2} \right\rfloor,
\end{align}
\begin{align}
\label{eq:mdef}
m=\lceil \sqrt{2\ell} \rceil,
\end{align}
\begin{align}
\label{eq:rdef}
r= \frac{k(k+1)}{2},
\end{align}
and define the integers $U$ and $V$ by
\begin{align}
\label{eq:UVdef}
U=\left \lfloor \frac{rN}{2q^{1/m}} \right\rfloor, \quad V=\left\lfloor  \frac{q^{1/m}}{2r} \right\rfloor.
\end{align}
We first note that
\begin{align}
\label{eq:inductionH}
UV\le \frac{N}{4}.
\end{align}
Let $1\le u \le U$ and $1\le v \le V$ be integers and write
\begin{align*}
\sum_{M< n \le M+N}e_q(an^{k})&=\sum_{M-uv < n \le M+N-uv}e_q(a(n+uv)^{k}) \\
&=\sum_{M < n \le M+N}e_q(a(n+uv)^{k}) \\ &  +\sum_{M-uv < n \le M}e_q(a(n+uv)^{k})-\sum_{M+N-uv < n \le M+N}e_q(a(n+uv)^{k}).
\end{align*}
Averaging the above over  $1\le u \le U$ and $1\le v \le V$ , using~\eqref{eq:inductionH} and applying  our induction hypothesis, we see that
\begin{align}
\label{eq:sW}
\left|\sum_{M< n \le M+N}e_q(an^{k})\right| \le \frac{|W|}{UV}+\frac{1}{2}\left(\frac{q^{1/(k-2)^{1/2}}}{N}\right)^{2/k(k+1)}Nq^{\varepsilon},
\end{align}
where
\begin{align*}
W=\sum_{1\le u \le U}\sum_{1\le v \le V}\sum_{M < n \le M+N}e_q(a(n+uv)^{k}).
\end{align*}
\par
We have
\begin{align*}
|W| \le \sum_{1\le v \le V}\left|\sum_{1\le u \le U}\sum_{M<n \le M+N}e_q\left(a\sum_{j=0}^{k}\binom{k}{j}n^{j}u^{k-j}v^{k-j}\right)\right|,
\end{align*}
hence by H\"{o}lder's inequality 
\begin{align*}
|W|^{\ell}
&\le V^{\ell-1}\sum_{1\le v \le V}\left|\sum_{1\le u \le U}\sum_{M<n \le M+N}e_q\left(a\sum_{j=0}^{k}\binom{k}{j}n^{j}u^{k-j}v^{k-j}\right)\right|^{\ell},
\end{align*}
so that for some complex numbers $\theta_v$ with $|\theta_v|=1$ we have
\begin{align*}
|W|^{\ell}
\le V^{\ell-1}\sum_{\substack{1\le u_i \le U \\ 1\le i \le \ell}}\sum_{\substack{M<n_i \le M+N \\ 1\le i \le \ell}}\left|\sum_{1\le v \le V}\theta_ve_q\left(a\sum_{j=0}^{k-1}\binom{k}{j}(n_1^{j}u_1^{k-j}+\dots+n_{\ell}^ju_{\ell}^{k-j})v^{k-j}\right)\right|.
\end{align*}

Let 
$$W_0=\sum_{1\le u_i \le U}\sum_{M<n_i \le M+N}\left|\sum_{1\le v \le V}\theta_ve_q\left(a\sum_{j=0}^{k-1}\binom{k}{j}(n_1^{j}u_1^{k-j}+\dots+n_{\ell}^ju_{\ell}^{k-j})v^{k-j}\right)\right|,$$
so that 
\begin{align}
\label{eq:UBstep1}
|W|^{\ell}\ll V^{\ell-1}W_0.
\end{align}
Let $I(\lambda_0,\dots,\lambda_{k-1})$ denote the number of solutions to the system of congruences
\begin{align*}
\binom{k}{j}(n_1^{j}u_1^{k-j}+\dots+n_{\ell}^ju_{\ell}^{k-j})\equiv \lambda_j \mod q, \quad j=0,\dots,k-1,
\end{align*}
in variables $n_1,\dots,n_{\ell},u_1,\dots,u_{\ell}$ satisfying
$$M< n_1,\dots,n_{\ell} \le M+N, \quad 1\le u_1,\dots,u_{\ell}\le U.$$
The above implies that
\begin{align*}
W_0=\sum_{\lambda_0,\dots,\lambda_{k-1}=0}^{q-1}I(\lambda_0,\dots,\lambda_{k-1})\left|\sum_{1\le v\le V}\theta_ve_q\left(a\sum_{j=0}^{k-1}\lambda_j v^{k-j}\right)\right|.
\end{align*}
 Two applications of H\"{o}lder's inequality gives
\begin{align}
\label{eq:W0b}
\nonumber W_0^{2r} &\le  \left(\sum_{\lambda_0,\dots,\lambda_{k-1}=0}^{q-1}I(\lambda_0,\dots,\lambda_{k-1}) \right)^{2r-2}\left(\sum_{\lambda_0,\dots,\lambda_{k-1}=0}^{q-1}I(\lambda_0,\dots,\lambda_{k-1})^2 \right) \\ & \times \left(\sum_{\lambda_0,\dots,\lambda_{k-1}=0}^{q-1}\left|\sum_{1\le v\le V}\theta_v e_q\left (\sum_{j=0}^{k-1}\lambda_j v^{k-j}\right)\right|^{2r}\right).
\end{align}
The term
\begin{align*}
\sum_{\lambda_0,\dots,\lambda_{k-1}=0}^{q-1}\left|\sum_{1\le v\le V}\theta_v e_q\left(\sum_{j=0}^{k-1}\lambda_j v^{k-j}\right)\right|^{2r},
\end{align*}
is bounded by $q^{k}$ times the number of solutions to the system of congruences
$$v_1^{j}+\dots+v_{r}^{j}\equiv v_{r+1}^{j}+\dots+v_{2r}^{j} \mod q, \quad j=1,\dots,k,$$
in variables $v_1,\dots,v_{2r}$ satisfying
$$1\le v_1,\dots,v_{2r}\le V.$$
Recalling the choice of $V$ in~\eqref{eq:UVdef} and applying Lemma~\ref{lem:VMVTq}, we see that
\begin{align}
\label{eq:I0}
\sum_{\lambda_0,\dots,\lambda_{k-1}=0}^{q-1}\left|\sum_{1\le v\le V}\theta_v e_q\left(\sum_{j=0}^{k-1}\lambda_j v^{k-j}\right)\right|^{2r}\le q^{k}J_{r,k}(V;q)\le q^{k}V^{2r-mk+m(m-1)/2+o(1)}.
\end{align}
We have
\begin{align}
\label{eq:I1}
\sum_{\lambda_0,\dots,\lambda_{k-1}=0}^{q-1}I(\lambda_0,\dots,\lambda_{k-1})=(NU)^{\ell},
\end{align}
and the term
\begin{align*}
\sum_{\lambda_0,\dots,\lambda_{k-1}=0}^{q-1}I(\lambda_0,\dots,\lambda_{k-1})^2,
\end{align*}
is equal to the number of solutions to the system of congruences
\begin{align}
\label{eq:Is}
n_1^{j}u_1^{k-j}+\dots+n_{\ell}^ju_{\ell}^{k-j}\equiv n_{\ell+1}^{j}u_{\ell+1}^{k-j}+\dots+n_{2\ell}^ju_{2\ell}^{k-j} \mod q, \quad j=0,\dots,k-1,
\end{align}
in variables $n_1,\dots,n_{2\ell},u_1,\dots,u_{2\ell}$ satisfying
$$M< n_1,\dots,n_{2\ell} \le M+N, \quad 1\le u_1,\dots,u_{2\ell}\le U.$$
Recalling~\eqref{eq:elldef} and considering only equations corresponding to $j=0,\dots,2\ell-1$ in~\eqref{eq:Is}, we see that
\begin{align*}
\sum_{\lambda_0,\dots,\lambda_{k-1}=0}^{q-1}I(\lambda_0,\dots,\lambda_{k-1})^2\le I_{k,\ell}(N,U),
\end{align*}
where $I_{k,\ell}(N,U)$ is defined as in Lemma~\ref{lem:mv11}. Recalling~\eqref{eq:Nassumption}, we have
\begin{align}
\label{eq:I2}
\sum_{\lambda_0,\dots,\lambda_{k-1}=0}^{q-1}I(\lambda_0,\dots,\lambda_{k-1})^2\le I_{k,\ell}(N,U)\le (NU)^{\ell+o(1)}.
\end{align}
Combining~\eqref{eq:W0b}~\eqref{eq:I0},~\eqref{eq:I1} and~\eqref{eq:I2} gives
\begin{align*}
W_0^{2r}\le q^{k+o(1)}(NU)^{\ell(2r-1)}V^{2r-mk+m(m-1)/2},
\end{align*}
and hence by~\eqref{eq:UBstep1}
\begin{align*}
|W|^{\ell}\ll q^{k/2r+o(1)}V^{\ell-mk/2r+m(m-1)/4r}(NU)^{\ell(1-1/2r)}.
\end{align*}
Combining the above with~\eqref{eq:sW} gives
\begin{align*}
\left|\sum_{M< n \le M+N}e_q(an^{k})\right| &\le \left(\frac{q^{k}}{V^{mk-m(m-1)/2}}\right)^{1/2r\ell}\frac{N^{1-1/2r}q^{o(1)}}{U^{1/2r}} \\ & \quad \quad \quad \quad +\frac{1}{2}\left(\frac{q^{1/(k-2)^{1/2}}}{N}\right)^{2/k(k+1)}Nq^{\varepsilon},
\end{align*}
which on recalling the choice of $U$ and $V$ in~\eqref{eq:UVdef} the above simplifies to
\begin{align*}
\left|\sum_{M< n \le M+N}e_q(an^{k})\right|  &\le q^{((m-1)/2\ell+1/m)/2r+o(1)}N^{1-1/r} \\ & \quad+\frac{1}{2}\left(\frac{q^{1/(k-2)^{1/2}}}{N}\right)^{2/k(k+1)}Nq^{\varepsilon}.
\end{align*}
Recalling the choice of $\ell,$ $m$ and $r$ in~\eqref{eq:elldef},~\eqref{eq:mdef} and~\eqref{eq:rdef} we get
\begin{align*}
\left|\sum_{M< n \le M+N}e_q(an^{k})\right| \le & \left(\frac{q^{1/(k-2)^{1/2}}}{N}\right)^{2/k(k+1)}Nq^{o(1)} \\ &+\frac{1}{2}\left(\frac{q^{1/(k-2)^{1/2}}}{N}\right)^{2/k(k+1)}Nq^{\varepsilon},
\end{align*}
and hence 
\begin{align*}
\left|\sum_{M< n \le M+N}e_q(an^{k})\right| \le \left(\frac{q^{1/(k-2)^{1/2}}}{N}\right)^{2/k(k+1)}Nq^{\varepsilon},
\end{align*}
by taking the term $o(1)$ in $q^{o(1)}$ to be sufficiently small.

\end{document}